\documentclass[reqno,11pt]{amsart}

\usepackage{amsmath}
\usepackage{amssymb}
\usepackage{amsthm} 
\usepackage{units} 
\usepackage[vcentermath]{youngtab}

\usepackage[english]{babel} 
\usepackage[utf8]{inputenc}
\usepackage{verbatim} 
\usepackage{soul} 
\usepackage{color,hyperref}

\usepackage{tikz}
\usetikzlibrary{matrix,arrows} 

\newtheorem{theorem}{Theorem}[section]
\newtheorem{lemma}[theorem]{Lemma}
\newtheorem{proposition}[theorem]{Proposition}

\newtheorem{conjecture}[theorem]{Conjecture}
\newtheorem{definition}[theorem]{Definition}
\theoremstyle{definition}

\newtheorem{remark}[theorem]{Remark}

\newcommand{\C}{\mathbb{C}}
\newcommand{\R}{\mathbb{R}}

\newcommand{\Z}{\mathbb{Z}}

\newcommand{\x}{\displaystyle}

\newcommand{\wt}{\widetilde}
\newcommand{\rx}{\mathcal{R}_X}

\renewcommand{\ul}{\underline{\lambda}}

\newcommand{\M}{\mathcal{M}} 

\newcommand{\CC}{\mathbb{C}}
\newcommand{\RR}{\mathbb{R}}
\newcommand{\NN}{\mathbb{N}}

\newcommand{\ZZ}{\mathbb{Z}}

\renewcommand{\ul}{\lambda}

\newcommand{\beee}{\begin{equation*}}
\newcommand{\bee}{\begin{equation}}
\newcommand{\ee}{\end{equation}}
\newcommand{\eee}{\end{equation*}}
\newcommand{\tq}{\: | \:}

\begin{document}

\title[Geometric realizations and duality for DM and DPV modules]{Geometric realizations and duality \\for Dahmen-Micchelli modules \\and De Concini-Procesi-Vergne modules}

\author[F. Cavazzani]{Francesco Cavazzani$^\ast$}
\thanks{$^\ast$Harvard University.}
\address{Department of Mathematics, FAS, Harvard University, 1 Oxford Street, Cambridge MA, 02138 USA}
\email{franc@math.harvard.edu}

\author[L.~Moci]{Luca Moci$^\dag$}
\thanks{$^{\dag}$IMJ, Universit\'{e} de Paris 7. Supported by a Marie Curie Fellowship of Istituto Nazionale di Alta Matematica.}
\address{IMJ, Universit\'{e} de Paris 7, 5 rue Thomas Mann, 75205 Paris Cedex 13, France}
\email{lucamoci@hotmail.com}
\maketitle
\thispagestyle{empty}
\begin{abstract}
We give an algebraic description of several modules and algebras
related to the vector partition function, and we prove that they can be realized as the
equivariant K-theory of some manifolds that have a nice
combinatorial description.
We also propose a more natural and general notion of duality between
these modules, which corresponds to a Poincaré duality-type
correspondence for equivariant K-theory.
\end{abstract}

\section{Introduction}
In recent years, the \emph{multivariate spline} and the \emph{vector partition function} have been studied by several authors. While the former is an essential tool in Approximation Theory, the latter has been studied in Combinatorics at least since Euler. Although they may seem quite different in nature, they can be viewed as the volume, and the number of integer points respectively, of a variable polytope; thus the partition function is the \emph{discretization} of the spline. In their book \cite{li}, De Concini and Procesi brought these functions to the attention of geometers, showing their relation with \emph{hyperplane arrangements} and \emph{toric arrangements}.

These functions are \emph{piecewise} polynomial/quasi-polynomial respectively, meaning that their support can be divided in regions called \emph{big cells}, such that on every big cell $\Omega$, the spline agrees with a polynomial $p_\Omega$, and the partition function agrees with a quasi-polynomial $q_\Omega$.

The polynomials $p_\Omega$, together with their derivatives, form a vector space $D(X)$. In a more algebraic language, $D(X)$ is generated by the elements $p_\Omega$ as module over the ring of polynomials, acting as derivations. On the other hand, the elements $q_\Omega$ generate $DM(X)$, as a module over the Laurent polynomials, acting as difference operators. The \emph{Dahmen-Micchelli modules} (\emph{DM modules} for short) $D(X)$ and $DM(X)$ are defined by a system of differential/difference equations respectively, having a simple combinatorial description in terms of the \emph{cocircuits} of the associated matroid (\cite{DM2, DMi, DH}). 

For every big cell $\Omega$, it is also natural to consider $D_\Omega(X)$, the cyclic submodule of $D(X)$ generated by the single polynomial $p_\Omega$, and $DM_\Omega(X)$, the cyclic submodule of $DM(X)$ generated by the quasi-polynomial $q_\Omega$. We call them the \emph{local DM modules}.

Furthermore, studying the partition function and the spline led to the definition of two flags of modules $\tilde{\mathcal F}_i(X)$ and $\tilde{\mathcal G}_i(X)$, of which $DM(X)$ and $D(X)$ are the smallest elements (\cite{DPV1, DPV2}). We call $\tilde{\mathcal F}_i(X)$ and $\tilde{\mathcal G}_i(X)$ the \emph{De Concini-Procesi-Vergne modules} (or \emph{DPV modules} for short).

The DM modules and the local DM modules naturally come together with their \emph{dual modules} $D^*(X)$, $D_\Omega^*(X)$, $DM^*(X)$,   $DM_\Omega^*(X)$, all described as quotients of the ring of polynomials by suitable ideals, and hence endowed by a structure of \emph{algebras}. As we show in Theorem \ref{resume}, this duality can be viewed in a more natural way via an $Ext$ functor. 
This choice has several advantages: it gives rise to a genuine duality in the category of finitely generated  $S[\frak g^*]$-modules (or $R(G)$-modules respectively), which corresponds to the Verdier duality in the derived categories. Furthermore, this algebraic duality leads to a statement of Poincar\'e duality in equivariant K-theory that we will prove in Theorem \ref{teor-F^*}. This approach
also allows us to define the \emph{dual DPV modules} $\mathcal F_i^*(X)$, $\mathcal G_i^*(X)$ (see Definitions \ref{def-Fi*}, \ref{def-Gi*} and Proposition \ref{DPV-filtr}).

Surprisingly, the modules and algebras above appear as invariants of geometric objects. In particular, the module $D(X)$ can be ``geometrically realized'' as the \emph{equivariant cohomology} of a differentiable manifold, while its discrete counterpart $DM(X)$ can be ``geometrically realized'' as the \emph{equivariant K-theory} of the same manifold.

The construction goes as follows. Let $X$ be the list of vectors in $\Z^d$ that defines the spline and the partition function. Each element of $X$ defines a 1-dimensional representation of the torus $G=(\mathbb S^1)^d$, hence there is a representation $M_X$ which is the direct sum of all them. Let $M_X^{fin}$ be the open subset of $M_X$ of points with finite stabilizer; this is the complement of a linear subspace arrangement. We have:

\begin{itemize}
\item[a)] $H^*_{c,G}(M_X^{fin})\cong D(X);$
\item[b)] $H^*_{G}(M_X^{fin})\cong D^*(X).$
\end{itemize}

Here the duality between $D(X)$ and $D^*(X)$ is realized by the Poincaré duality between the ordinary cohomology $H$ and compact support cohomology $H_{c}$ of the \emph{toric orbifold} $M_X^{fin}/G_\C$ (see Section \ref{toric}).

An analogue of this result can be provided for discrete DM modules; this is done by considering {equivariant K-theory}. However, this is intrinsically a compact support cohomology theory, thus lacking of a natural non-compact-support counterpart. Instead of looking for a different definition of equivariant K-theory non involving compact support, our approach will be based on compactifying the manifold $M_X^{fin}$. In fact we intersect it with the equivariant unit sphere and then remove open tubular neighborhoods of every resulting hypersurface, thus obtaining a compact \emph{manifold with corners} $S_X^{fin}$. Then we have:
\begin{itemize}
\item[c)] $ K_G^*(M_X^{fin})\cong DM(X);$
\item[d)] $K_G^*(S_X^{fin})\cong DM^*(X).$
\end{itemize}

Facts a), b) and c) are consequences of statements proved by De Concini, Procesi and Vergne for DPV modules, that we will recall in Theorems \ref{teor-F}, \ref{teor-G}, \ref{teor-G^*}. In particular the proof of a), given in \cite{DPV2}, is based on \emph{index theory of transversally elliptic operators}, and answers to a question that Atiyah raised about four decades ago (\cite{ell}). The proof of c), given in \cite{DPV3}, required the development of a further tool, the \emph{infinitesimal index}. 

In this paper we prove Theorem \ref{teor-F^*} for dual DPV modules, that implies fact d).

We leave open the problem of finding geometric realizations for the local modules $D_\Omega(X)$ and $D_\Omega(X)$ (see Conjectures \ref{cong-local} and \ref{cong-local2}).

We think that also other objects, such as some modules arising from the \emph{power ideals} studied by Ardila and Postnikov \cite{AP}, the \emph{semi-internal zonotopal} modules introduced by Holtz, Ron and Xu in \cite{HRX}, and their generalizations proposed by Lenz \cite{Le}, may admit a similar geometric realization, that we hope to study in future papers.

\bigskip

\textbf{Acknowledgements. }
We are very grateful to Corrado De Concini and Michèle Vegne for many inspiring suggestions and conversations.
We also want to thank Dave Anderson, Omar Antol\'in Camarena, Alessandro D'Andrea, Alex Fink, Gijs Heuts, Mike Hopkins, Matthias Lenz, Ivan Martino and Angelo Vistoli for helpful remarks and suggestions. 

\tableofcontents

\section{Recalls on Representation Theory and Combinatorics}

\subsection{First notations}\label{notations}

Let $G$ be an abelian compact Lie group, and let $\Gamma\doteq Hom(G,\mathbb S^1)$ be its character group. 
We will assume for simplicity that $G$ is connected, i.e., it is a torus.
Let $\mathfrak{g}$ be the Lie algebra of $G$: this can be identified to the tangent space at $1$ of $G$, so that the differential $d\lambda$ of every character $\lambda$ is an element of the dual space $\mathfrak{g}^*$; by a slight abuse of notation we will sometimes identify $d\lambda$ with $\lambda$.
We denote by $S[\mathfrak g^*]$ the symmetric algebra of polynomial functions on $\mathfrak g$, and by $R(G)$ the character ring of $G$, that is, the group algebra of $\Gamma$.

For the sake of concreteness, let us say that $\mathfrak{g}^*$ (as well as $\mathfrak{g}$) is a real vector space $V$, and if we denote its dimension by $d$, $G$ is isomorphic to $(\mathbb S^1)^d$, while $\Gamma$ is isomorphic to $\Z^d$. $S[\mathfrak g^*]$ is isomorphic to the ring of polynomials $\R[x_1, \dots, x_d]$ while $R(G)$ is isomorphic to the ring of Laurent of polynomials $\Z[x_1^{\pm 1}, \dots, x_d^{\pm 1}]$. Then $\Gamma$ embeds in $S[\mathfrak g^*]$ and in $R(G)$ as
$$(m_1, \dots, m_d)\mapsto m_1 x_1+ \dots +m_d x_d$$
and
$$(m_1, \dots, m_d)\mapsto x_1^{m_1} \dots x_d^{m_d}$$
respectively.

All the results in this paper may be extended to the case of $G$ not being connected.
In this case $G$ is isomorphic to the product of a compact torus $(\mathbb S^1)^d$ and a finite group $G_f$, and $\Gamma$ is isomorphic to $\Z^d\times G_f$. Then we have a projection $\Gamma\to V$ which forgets the torsion part of $\Gamma$.

Let $\mathcal{C}[\Gamma]$ be the space of $\Z$-valued functions on $\Gamma$. On this space every element $a\in\Gamma$ acts as the translation $\tau_{a}$; this extends to an action of $R(G)$. We define the difference operator $\nabla_a=1-\tau_a$, i.e.:
$$\nabla_a f(x)\doteq f(x)-f(x-a).$$

Let $X=[a_1, \dots, a_n]$ be a finite list of elements of $\Gamma$. We will always assume that $rk (X)=rk(\Gamma)$, and that none of the elements of $X$ is zero (otherwise, it is simple to reduce to this case).

Following \cite{DPV1}, we say that a linear subspace $\underline{s}\subseteq V$ is \textit{rational} if it is spanned by elements of $X$;
we denote by $\rx$ the set of all rational subspaces.
With a little abuse of notation, we will indicate by $\underline{s}$ both a linear subspace of $V$ and the list of elements of $X$ that belong to such subspace.

\subsection{DM modules}

We recall that $A\subset X$ is a \emph{cocircuit} if $A=X\setminus \underline s$ for some rational hyperplane $\underline s$ (i.e. for some $\underline s \in \rx$ such that $codim(s)=1$).
Let us define the set
$$\mathcal L(X)\doteq \{A\subseteq X \tq C\subseteq A \text{ for some cocircuit } C\}.$$
For every $A\subseteq X$, we consider the difference operator $\nabla_A=\prod_{a\in A}\nabla_a$ acting on $\mathcal C[\Gamma]$.
The \emph{discrete DM module} is defined as
$$DM(X)=\{f\in\mathcal{C}[\Gamma]\tq \nabla_{A}f=0 \text{ for every } A\in\mathcal L(X) \}.$$

This can be seen as the ``discretization'' of the \emph{differentiable DM module}
$$D(X)=\{f: V\rightarrow \RR \tq \partial_{A}f=0 \text{ for every } A\in \mathcal L(X)\}$$
where the differential operator
$\partial_A=\prod_{a\in A}\partial_a$
is just the product of directional derivatives.
Of course, as in many of the definitions that will follow, it is enough to check the equations above for the minimal elements of $\mathcal L (X)$, that is, the cocircuits.

The space $D(X)$ is naturally a module over $S[\frak g^*]\simeq \R[x_1, \dots, x_d]$, the action being given by derivation:
$x_i\cdot f\doteq\partial_{x_i}p$.
On the other hand, $DM(X)$ is naturally a module over $R(G)\simeq \Z[x_1^{\pm 1}, \dots, x_d^{\pm 1}]$, the action being given by difference operators:
$x_i\cdot f\doteq\nabla_{x_i}f$.

The  vector space underlying to $D(X)$ is the space $\mathcal D(X)$ studied in \cite{HR}.

We say that the list $X$ is \emph{unimodular} (or \emph{totally unimodular}) if every $k\times k$ minor of $X$ is equal to $1$, $-1$ or $0$, for every $1\leq k\leq d$; in other words, if every basis taken out from $X$ generates the whole lattice $\Gamma$. In this case it is easy to see that $D(X)=DM(X)\otimes_\Z \R$.

\begin{remark}
One can see from the definition that $D(X)$ essentially depends on the linear algebra of the vectors in $X$, while $DM(X)$ also depends on the arithmetic of the vectors. The need to encode this arithmetic information in a combinatorial object led to the introduction of \emph{arithmetic matroids} (\cite{DM, BM}) and \emph{matroids over $\Z$} (\cite{FM}).
\end{remark}

As proved in \cite{DM1} and \cite{DM2} respectively, the dimension of $D(X)$ (as a real vector space) is equal to the number of bases which can be extracted from $X$, while the rank of $DM(X)$
 (as a free $\Z$-module) is equal to the volume of the \emph{zonotope}

$$\mathcal Z(X)=\left\{\sum_{i=1}^d t_i a_i, \qquad 0\leq t_i\leq 1\right\}.$$

These spaces were introduced in order to study two important functions, that we are going to describe in the next subsection.

\subsection{Vector partion function and multivariate spline}

For every $\lambda\in\Gamma$, we define $\mathcal{P}_X(\lambda)$ as the number of ways we can write

$$
\lambda=\sum_{i=1}^n x_i a_i\qquad x_i\in\NN.
$$
Since we want this number to be finite, we assume that all the elements $a_i$ of the list $X$ lie on the same side of a hyperplane in $V$. We can always do that, eventually replacing some vectors by their opposites.
We call $\mathcal{P}_X(\lambda)$ the \emph{vector partition function}, or simply the \emph{partition function}.

The equation above is indeed a system of diophantine equations, one for every coordinate of $\lambda$. We can rewrite this system as $\overline X x =\lambda$, where $\overline X$ is the matrix whose columns are the vectors $a_i\in X$, and $x$ is the vector whose entries are the variables $x_i$.

This defines a subspace of $\RR^n$. The intersection of this subspace with the positive orthant is a \emph{variable polytope}

$$
P_X(\lambda)=\left\{x\in\left(\RR_{\geq 0}\right)^n\mid \overline X x=\lambda\right\}$$
and the partition function is the number of its integer points:
$$
\mathcal{P}_X(\lambda)=\left\vert P_X(\lambda)\cap\ZZ^n\right\vert.
$$

Then $\mathcal{P}_X(\underline{\lambda})$ is related with another function:
$\M_X(\ul)\doteq vol\left(P_X(\ul)\right)$.
Indeed the number of integer points of a polytope is the ``discrete analogue'' of its volume.
The function $\M_X(\ul)$, which is well defined for every $\ul\in \RR^n$, is known as the \emph{multivariate spline} (or simply the \emph{spline}).
Splines are used in Numerical Analysis to approximate functions.
The word ``spline'' means that $\M_X$ is piecewise polynomial and ``as smooth as possible''; we will now make more precise this statement.

First of all, notice that both the functions $\M_X, \mathcal P_X$ are supported on the cone

$$C(X)=\left\{\sum_{i=1}^n t_i a_i, \qquad t_i\geq 0\right\}$$

For every cocircuit $A$, we consider the cone $C(X\setminus A)$ spanned by $X\setminus A$. We define a \emph{big cell} as a connected component of
$$C(X)\setminus \bigcup_{A\in \mathcal L(X)}C(X\setminus A).$$

Then we have:
 \begin{theorem}[de Boor-Hollig]
For every big cell $\Omega$, there is a polynomial $p_\Omega\in D(X)$ such that $\M_X$ and  $p_\Omega$ coincide on $\Omega$. Moreover, all the polynomials $p_\Omega$ have degree $n-d$ and $\M_X$ is continuous of class $C^{n-d-1}(V)$.

\end{theorem}

Since $\mathcal P_X$ is the discretization of $\M_X$, it is natural to wish a discrete analogue of the theorem above. Such an analogue exists, but with two important differences.

The first can be understood by looking at the 1-dimensional example given by the list $[2,1]\in \Z$. Here $\mathcal P_X(\ul)=\ul/2 +1$ when $\ul$ is even, and $\mathcal P_X(\ul)=\ul/2 +1/2$ when $\ul$ is odd. Then, in general,  we recall the following definition: a function $q: \Gamma\rightarrow \Z$ is a \emph{quasi-polynomial} if there exist a finite index subgroup of $\Gamma$ such that on every coset, $q$ coincides with a polynomial.
As we will see, the partition function is piecewise quasi-polynomial.

The second issue is what a ``discrete analogue'' of smoothness should be. The natural idea is to require that the regions of quasi-polynomiality overlap a bit, that is, given two neighboring big cells, there is a stripe on which the corresponding quasi-polynomials agree.
More precisely, by \cite{SzV} we have:

\begin{theorem}[Dahmen-Micchelli, Szenes-Vergne]\label{DMTH}
For every big cell, there is a quasi-polynomial $q_\Omega\in DM(X)$ such that $\mathcal{P}_X$ and $q_\Omega$ coincide on $\Omega$. Moreover, they coincide on a larger region, the Minkowski sum of $\Omega$ and $-\mathcal{Z}(X)$.
\end{theorem}

The two theorems above motivate the interest for the modules $D(X)$, $DM(X)$. In fact these modules contain the ``local pieces'' $p_\Omega$, $q_\Omega$ respectively; more precisely, $D(X)$ is the $S[\frak g^*]-$module generated by the polynomials $p_\Omega$,  and $DM(X)$ is the $R(G)-$module generated by the quasi-polynomials $q_\Omega$, where $\Omega$ ranges over all the big cells.

\subsection{Local DM modules and their generalizations}\label{loc-gen}

For every big cell $\Omega$, it is also natural to consider $D_\Omega(X)$, the cyclic submodule of $D(X)$ generated by the polynomial $p_\Omega$, and $DM_\Omega(X)$, the cyclic submodule of $DM(X)$ generated by the quasi-polynomial $q_\Omega$. 
These modules admit a simple combinatorial description. Let us define

$$\mathcal L_\Omega(X) \doteq\{A\subseteq X \tq C(X\setminus A)\nsupseteq \Omega\}.$$

Notice that $\mathcal L_\Omega(X) $ contains all the cocircuits and is closed under taking supsets. We have:

$$D_\Omega(X)=\{f: V\rightarrow \RR \tq \partial_{A}f=0 \:\forall A\in \mathcal L_\Omega(X)  \}$$

$$DM_\Omega(X)=\{f\in\mathcal C[\Gamma]\tq\nabla_{A}f=0 \:\forall A \in \mathcal L_\Omega(X) \}$$

More generally, given any subset $\mathcal{T}\subseteq 2^X$ closed under taking supsets (that is, if $A\in\mathcal{T}$ and $A\subset B\in 2^X$, then $B\in\mathcal{T}$), we can consider
the $S[\mathfrak g^*]$-module
 $$D_{\mathcal T}(X)=\{f: V\rightarrow \RR \tq \partial_A f=0\:\forall A \in \mathcal T\}$$
and the $R(G)$-module
$$DM_{\mathcal T}(X)=\{f\in\mathcal C[\Gamma]\tq \nabla_A f=0\:\forall A \in \mathcal T\}$$

\begin{lemma}\label{fin-dim}
The module  $DM_{\mathcal T}(X)$ has finite rank over $\Z$ if and only if $\mathcal T$ contains all the cocircuits.
The module $D_{\mathcal T}(X)$ has finite dimension over $\RR$ if and only if $\mathcal T$ contains all the cocircuits.

\end{lemma}
\begin{proof}
If $\mathcal T$ contains all the cocircuits, then $DM_{\mathcal T}(X)\subseteq DM(X)$, because it is defined by the same difference equations, plus further ones. 

Now let us assume that $\mathcal T$ does not contain a cocircuit $A$, and let $\underline s$ be a rational hyperplane such that $A=X\setminus \underline s$. Then we will show that any function that is constant on $\underline{s}$ and on all its translates belongs to $DM_{\mathcal{T}}(X)$. These functions are annihilated by all the operators $\nabla_a$ with $a\in\underline{s}$, and then by all the $\nabla_B$ for all $B$ such that $B\nsubseteq A$; now, notice that all elements of $\mathcal T$ satisfy this condition, so that all these functions belong to $DM_{\mathcal{T}}(X)$. Now, being $\underline{s}$ a proper subspace, its cosets in $\Gamma$ are infinitely many, so such functions are a subset of $DM_{\mathcal{T}}(X)$ of infinite rank over $\Z$ (in fact, a basis of it consists of uncountably many elements).

The same proof holds for $D_{\mathcal T}(X)$.
\end{proof}

\subsection{DPV modules}
Altough $DM(X)$ contains all the local pieces $q_\Omega$ of the partition function $\mathcal P_X$, it does not contain $\mathcal P_X$ itself. In fact all the elements of $DM(X)$ are genuine quasi-polynomials, while the partition function is \emph{piecewise} quasi-polynomial. It is then desirable to have a space that contains both $DM(X)$ and $\mathcal P_X$. This is the $\Z$-module
$${\mathcal{F}}(X)=\{f\in\mathcal{C}[\Gamma]\tq \nabla_{X\backslash\underline{s}}f\ \text{is supported on}\ \underline{s}\ \:\forall\underline{s}\in\rx\}.$$
By ``supported on $\underline s$'' we mean that the support of $f$ (i.e., the subset of the domain on which $f$ takes nonzero values) is contained in $\underline s$.
This space comes with a natural filtration

$$DM(X)={\mathcal{F}}_d(X)\subset{\mathcal{F}}_{d-1}(X)\subset\ldots\subset{\mathcal{F}}_0(X)={\mathcal{F}}(X)$$
where

$${\mathcal{F}}_i(X)=\left\{f\in\mathcal{C}[\Gamma]\tq \begin{array}{l}
\nabla_{X\backslash\underline{s}}f=0 \ \text{if }dim(\underline{s})<i\\
\nabla_{X\backslash\underline{s}}f\ \text{is supported on }\ \underline{s}\ \text{otherwise} \end{array}\ \right\}.$$

Clearly, these $\Z$-modules are \emph{not} invariant for the action by translations of $R(G)$ (because after a translation $\nabla_{X\backslash\underline{s}}f$ is going to be supported on a translate of $\underline{s}$). In fact they generate the \emph{discrete De Concini-Procesi-Vergne modules} (\emph{discrete DPV modules} for short)

$$\tilde{\mathcal{F}}(X)=\{f\in\mathcal{C}[\Gamma]\tq \nabla_{X\backslash\underline{s}}f\ \text{is supported on a finite number of translates of}\ \underline{s}\ \:\forall\underline{s}\in\rx\}$$
$$\tilde{\mathcal{F}}_i(X)=\left\{f\in\mathcal{C}[\Gamma]\tq \begin{array}{l}
\nabla_{X\backslash\underline{s}}f=0 \ \text{for every rational proper subspace }\underline{s}\ \text{of dimension }<i\\
\nabla_{X\backslash\underline{s}}f\ \text{is supported on a finite number of translates of}\ \underline{s}\ \text{otherwise} \end{array}\ \right\}.$$

The space ${\mathcal{F}}(X)$ is a free $\Z$-module of rank equal to the number of integer points of the zonotope $\mathcal Z(X)$. On the other hand, $\tilde{\mathcal{F}}(X)$ and $\tilde{\mathcal{F}}_i(X)$ have clearly infinite rank over $\Z$. All these spaces have been introduced and studied in \cite{DPV1, DPV2}.

In the same way, we can define a real vector space $\mathcal G(X)$ containing both the multivariate spline $\mathcal M_X$ and the space $D(X)$, with a filtration $\mathcal G_i(X)$. The definitions are exactly the same, except for $\nabla$ that is replaced by $\partial$. Again, these vector spaces generate the \emph{differentiable DPV $S[\mathfrak g^*]$-modules} $\tilde{\mathcal{G}}(X)$,  $\tilde{\mathcal{G}}_i(X)$ respectively (see \cite{DPV4}).
These modules have a natural grading, which is given  by the degree of the homogeneous polynomials (or piecewise polynomials).

\begin{remark}
Interestingly, the \emph{external zonotopal space} $\mathcal{D}_+(X)$ studied in \cite{HR} has the same dimension as $\mathcal G(X)$ over $\R$. When the list $X$ is unimodular, this dimension is equal to the number of integer points of the zonotope $\mathcal Z(X)$. However, $\mathcal{D}_+(X)$ and $\mathcal G(X)$ are different spaces: the former contains polynomials, the latter \emph{distributions}. It is not surprising, then, that only the first space is closed under derivations, i.e., is an $S[\mathfrak g^*]$-module. It would be interesting, nevertheless, to establish some canonical correspondence between these two spaces.

The same considerations hold for the \emph{semi-external zonotopal spaces} $\mathcal{D}_+(X, \mathbb I_i)$ studied in \cite{HRX},
when $\mathbb I_i$ is the family of linearly independent sublists of $X$ of rank at least $i$ ($0\leq i\leq d$). In fact, these spaces have the same dimension as the spaces $\mathcal G_i(X)$.

\end{remark}

\subsection{Some representations}\label{RT}

For every rational subspace $\underline{s}\in\rx$, let us define
$$G_{\underline{s}}\doteq \{g\in G \: | \: \chi(g)=1\:\forall \chi\in \Gamma\cap \underline{s}    \}.$$
Notice that this is a torus of dimension $dim(G_{\underline s})= codim (\underline s)$; one can split in a non-canonical way $G$ as $G\simeq G_{\underline s} \times G/G_{\underline s}$.

To every element $a$ of $\Gamma$ correspond a 1-dimensional representation $M_a$ of $G$,
on which every $g\in G$ acts as the multiplication by the scalar $a(g)$.
Then for every list $X$ of elements of $G$ we consider the representation $$\x M_X=\bigoplus_{a\in X}M_a.$$
which has dimension $n$ over $\CC$.

Given a $G$-invariant Hermitian product on $M_X$, let $S_X$ be the unit sphere.
This is a $G$-manifold of dimension $2n-1$ over $\R$.

For every $A\subset X$ we can consider the coordinate subspace $M_A\subset M_X$ given by

$$M_{A}=\{(z_a)_{a\in X}\in M_X\tq z_a=0\text{ for every }a\notin A\}.$$

Notice that the subgroup of elements of $G$ that send the subspace $M_{\underline{s}}\doteq M_{X\cap\underline{s}}$ in itself is precisely $G_s$.

Given any subset $\mathcal{T}\subseteq 2^X$ closed under taking supsets, we can consider

$$M_X^{\mathcal{T}}=M_X\setminus\bigcup_{A\in\mathcal{T}}M_{X\setminus A}.$$

In particular when $\mathcal T=\mathcal L(X)$, we get

$$ M_X^{\mathcal{L}(X)}=M_X\setminus\bigcup_{A\in\mathcal{L}(X)}M_{X\setminus A}=
M_X\setminus\bigcup_{\underline{s}\in\rx}M_{\underline{s}}=
M_X^{fin }.$$
Notice that in the formula above it is sufficient to take the first union over the minimal elements (i.e. cocircuits) and the second union over the maximal elements (i.e. rational hyperplanes).

In the same way, we can stratify $M_X$ by dimension of orbits, and for every $	i\leq d$ call $M^i_X$ the subset of points having $i$-dimensional orbit. Then the subset $M_X^{\geq i}$ of points whose orbit is $i$-dimensional or more is
\begin{equation}\label{tolgochiusi}
M_X^{\geq i}=M_X\setminus\bigcup_{\tiny \begin{array}{c} \underline{s}\in\rx\\ \text{dim}(\underline{s}) < i\end{array}}M_{\underline{s}}.
\end{equation}
Then we have:

$$M_X^{fin}=M_X^{\geq d}\subseteq M_X^{\geq d-1}\subseteq\dots\subseteq M_X^{\geq 0}=M_X$$

Finally, when $\mathcal T=\mathcal L_\Omega(X)$ (as defined in Section \ref{loc-gen}), we get

$$M_X^\Omega=\{(z_a)_{a\in X}\in M_X\tq \exists A\subset X \tq C(A)\supseteq \Omega\text{ and } z_a\neq 0 \:\forall a\in A\}.$$

\subsection{An example}\label{example}

Let us take the list $X=[(1,0), (0,1), (k,k)]$ in $\Gamma=\Z^2$, where $k$ is a positive integer. 
Then we have two big cells; let us call $\Omega$ the one whose extremal rays are spanned by the vectors $(1,0)$ and $(k,k)$ and $\Omega'$ the other one.

The torus $G_\C=(\mathbb S^1)^2$ 
acts on $M_X=\C^3$ by $(t,s).(z_1, z_2, z_3)=(tz_1, sz_2, t^ks^kz_3)$.
Since the cocircuits are the three couples of vectors, we have that
$$M_X^{fin}=\{(z_1, z_2, z_3)\in M_X \tq z_1z_2\neq 0 \text{ or }  z_1z_3\neq 0 \text{ or } z_2z_3\neq 0 \}.$$

If $k=1$, $D(X)$ and $DM(X)$ have a basis given by the three functions $x$, $y$ and $1$, over $\R$ and $\Z$ respectively. A basis of the local modules corresponding to $\Omega$ is given by $y$ and $1$, while a basis of the local modules corresponding to $\Omega'$ is given by $x$ and $1$.

If $k>1$, $D(X)$ is unchanged while $DM(X)$ is the free $\Z-$module of rank $2k+1=\text{vol } \mathcal Z(X)$ that is spanned by $x$, $y$ and by all the functions that are constant in one of the two variables and $k-$periodic in the other. 

Finally, the subset $M_X^\Omega$ is given by the condition $z_1z_2\neq 0 \text{ or }  z_1z_3\neq 0$, that is
$$M_X^\Omega=\{(z_1, z_2, z_3)\in M_X \tq z_1\neq 0 \text{ and }  (z_2, z_3)\neq (0, 0)\}$$

and the module 
$DM_\Omega(X)\simeq K^*_G(M_X^\Omega)$ has rank $k+1$ over $\Z$.

\section{Duality of modules}\label{duality}

In this section we describe some duality relations between DM and DPV modules, that will be reflected in duality statements in equivariant K-theory.

\subsection{Duality for DM modules}\label{duality1}
Let $\mathcal T\subseteq 2^X$ contain all the cocircuits, as in Lemma \ref{fin-dim}.
Let us consider the two embeddings of $\Gamma$ in $S[\mathfrak g^*]$ and in $R(G)$ described in Section \ref{notations}.
Then we can define an ideal of $S[\frak g^*]$ as
$$\mathcal J^\partial_\mathcal T\doteq (d_A)_{A\in \mathcal T}\text{, where }d_A\doteq\prod_{a\in A}a.$$

In the same way we can define an ideal of $R(G)$ as
$$\mathcal J^\nabla_\mathcal T\doteq (\nabla_A)_{A\in\mathcal{T}}\text{, where }\nabla_A\doteq \prod_{a\in A}(1-a).$$

Notice that these ideals are annihilators of $D_\mathcal{T}(X)$ as $S[\frak g^*]$-module and of $DM_\mathcal{T}(X)$ as $R(G)$-module respectively.
Then we define:
$$D_\mathcal T^*(X)\doteq S[\frak g^*] / \mathcal J^\partial_\mathcal T$$
and
$$DM_\mathcal T^*(X)\doteq R(G) / \mathcal J^\nabla_\mathcal T.$$
By definition these are a $S[\frak g^*]$-module and a $R(G)$-module respectively; but, unlike $D_\mathcal T(X)$
 and $DM_\mathcal T(X)$, they also have a multiplicative structure, i.e they are algebras.

In particular when $\mathcal T=\mathcal L(X)$ we get two algebras that we will denote by $D^*(X)$, $DM^*(X)$, and call the \emph{dual DM modules}. While when $\mathcal T= \mathcal L_\Omega(X) $ for a big cell $\Omega$, we will denote the two corresponding algebras by $D_\Omega^*(X)$, $DM_\Omega^*(X)$.

\begin{remark}
The ideal $\mathcal J^\partial_{\mathcal{L}(X)}$ has been studied in \cite{HR}, where is denoted by $\mathcal J(X)$.  The vector space underlying to $D^*(X)$ is the space therein denoted by $\mathcal P(X)$, which is defined using a \emph{power ideal} associated to $X$.
\end{remark}

\begin{lemma}\label{hom-d}
We have the isomorphism of $S[\mathfrak{g}^*]$-modules $$D_\mathcal T(X)\cong\text{Hom}_{\R}(D_\mathcal T^*(X),\R).$$
\end{lemma}
\begin{proof}
Clear by definition: the homomorphism $D_\mathcal T^*(X)\rightarrow \R$ are the homomorphisms $S[\frak g^*]\rightarrow \R$ that are zero on $\mathcal J^\partial_\mathcal T$, which correspond precisely to the polynomial functions on $V=\frak g^*$ that satisfy the defining conditions for  $D_\mathcal T(X)$.
\end{proof}

\begin{lemma}\label{hom-dm}
We have the isomorphism of $R(G)$-modules $$DM_\mathcal T(X)\cong\text{Hom}_{\Z}(DM_\mathcal T^*(X),\Z).$$
\end{lemma}
\begin{proof}
Since $\Gamma$ is a set of generators of $R(G)$ as $\Z$-module, we have an isomorphism of $R(G)$-modules
$$\mathcal C[\Gamma]\simeq\text{Hom}_\Z(R(G),\Z).$$
Then, we can see by definition $DM_\mathcal T(X)$ inside $\mathcal C[\Gamma]$ as the vanishing locus of the endomorphisms given by $\nabla_{A}$ for $A\in \mathcal T$; and we can see naturally as well $\text{Hom}_{\Z}(DM_\mathcal T^*(X),\Z)$ sitting inside $\text{Hom}_\Z(R(G),\Z)$ as vanishing locus of the same elements of $R(G)$, because $DM_\mathcal T^*(X)$ is the quotient of $R(G)$ by these elements.
\end{proof}

In the case of $DM(X)$, Lemma~\ref{hom-dm} has been already proved in \cite{li}, Proposition~13.16, and behind there is a deeper result, that is that the module $DM^*(X)$ is indeed torsion free.

\subsection{Ext functor and duality}\label{duality2}

From an algebraic point of view, the duality described above is not completely satisfactory. In fact, since we are dealing with $S[\frak g^*]$-modules and $R(G)$-modules, it would be more natural to have a duality involving these rings. Furthermore, in the case of DPV modules the attempt to build duals via the functors
$\text{Hom}_{\R}(\cdot,\R)$ and $\text{Hom}_{\Z}(\cdot,\Z)$ does not give good results, because these modules have infinite rank (over $\R$ and over $\Z$ respectively).

On the other hand, defining a duality via the functors $Hom_{S[\mathfrak{g}^*]}(\cdot,S[\mathfrak{g}^*])$ and $Hom_{R(G)}(\cdot,R(G))$ would not yield desiderable results neither: in fact, we have
$$Hom_{S[\mathfrak{g}^*]}(D^*_{\mathcal{T}}(X),S[\mathfrak{g}^*])=0$$
and
$$Hom_{R(G)}(DM^*_{\mathcal{T}}(X),R(G))=0,$$

as one can easily see by the fact that $S[\mathfrak{g}^*]$ and $R(G)$ are domains.

Hence we propose to take a more abstract perspective, realizing the duality via the functors $Ext^*_{S[\mathfrak{g}^*]}(\cdot,S[\mathfrak{g}^*])$ and $Ext^*_{R(G)}(\cdot,R(G))$.

The advantages of this choice will be multiple: it gives rise to a genuine duality in the category of finitely generated  $S[\frak g^*]$-modules (or $R(G)$-modules respectively), which corresponds to the Verdier duality in the derived categories. Furthermore, this algebraic duality is reflected in (and inspired by) the geometric duality the we will describe in the next Sections.
Finally, this notion allows to build duals also for the DPV modules.

We recall that $Ext^*$ is the collection of the $Ext^i$, which are the derived functors of $Hom$. In particular, they can be nonzero only for $0\leq i\leq d$ ($0\leq i\leq d+1$ in the discrete case) and $Ext^0$ is $Hom$ itself.

First of all, we check that we are actually extending the duality defined in the previous Subsection. In fact we have:

\begin{proposition}\label{ext-d}
Let $V$ be an $S[\mathfrak{g}^*]$-module that is finite dimensional as a vector space. Then we have the isomorphisms of $S[\mathfrak{g}^*]$-modules:
$$Ext^i_{S[\mathfrak{g}^*]}(V,S[\mathfrak{g}^*])=0\ \emph{, for } 0\leq i<d$$
$$Ext^{d}_{S[\mathfrak{g}^*]}(V,S[\mathfrak{g}^*])\simeq Hom_{\R}(V,\R).$$
\end{proposition}

\begin{proposition}\label{ext-dm}
Let $M$ be an $R(G)$-module that is free over $\Z$ and has finite rank over it. Then we have the isomorphisms of $R(G)$-modules:
$$Ext^i_{R(G)}(M,R(G))=0\  \emph{, for } 0\leq i<d \text{ and } i=d+1$$
$$Ext^{d}_{R(G)}(M,R(G))\simeq Hom_{\Z}(M,\Z).$$
\end{proposition}

We have not found these statements in literature, even though we believe it is a well known fact. Indeed, there are many analogues involving different categories, for instance \cite[Section 5]{BDK}; a general version implying Proposition \ref{ext-d} can be found in \cite[Theorem 7.0.5]{Ch}.

\begin{proof}
We are going to prove the two statements together. We will indicate by $R$ the ring $S[\mathfrak{g}^*]$ when dealing with the first statement and the ring $R(G)$ when dealing with the second, denoting by $x_1,\ldots, x_d$ the variables in both rings. In the same way, $\mathbb{F}$ will stand for $\R$ in the first case and for $\Z$ in the second. Let $\{e_1,\ldots,e_s\}$ be a basis of $M$ as a free $\mathbb{F}$-module, and let us denote by $M_R$ the module $R^s$, with a basis denoted by $\{p_1,\ldots, p_s\}$, and the surjective $R-$linear map $$M_R\xrightarrow{\delta_0} M$$ sending $p_i$ to $e_i$. Note that the action $\Psi$ of $R$ on $M$ is determined by specifying the $d$ matrices $\Psi(x_i)$; these have to be $s\times s$ matrices with coefficients in $\mathbb{F}$, and commuting one with each other. In the second (i.e. discrete) case, we also need them to be invertible in $GL_s(\Z)$.

These matrices also act on $M_R$; in particular, in $M_R$ the operators $x_i$ and $\Psi(x_i)$ are different, one acting as a constant, one as linear transformation. Furthermore, in the module $Hom_{\mathbb{F}}(M,\mathbb{F})$, denoting by $\{e_1^*,\ldots,e_s^*\}$ the $\mathbb{F}$-dual basis, the action of $R$ is given by the transpose matrices $^t\Psi(x_i)$.
Starting from $\delta_0$, we will create a free resolution of $M$ as $R$-module
$$0\to M_R^{\binom{d}{d}}\xrightarrow{\delta_d} M_R^{\binom{d}{d-1}}\to\ldots\xrightarrow{\delta_2}M_R^d\xrightarrow{\delta_1}M_R\xrightarrow{\delta_0}M\to 0,$$
in which we need to describe the maps $\delta_i:M_R^{\binom{d}{i}}\to M_R^{\binom{d}{i-1}}$; we are going to describe these maps block by block, using the decomposition
$$M_R^{\binom{d}{i}}=\bigoplus^{I\subset\{1,\ldots,d\}}_{|I|=i}(M_R)_I.$$
Given $I=\{b_1,\ldots, b_i\}$, the image of a vector $v$ of $(M_R)_I$ by $\delta_i$ is going to be zero on the sets $J\nsubseteq I$, while on $J=I\setminus\{b_j\}$ the coordinate is going to be
$$(-1)^{j+1}(x_{b_j}v-\Psi(x_{b_j})v).$$
The reason of this choice is, at first, to have the image of $\delta_1$ in $M_R$ generated by all the vectors of the form $x_{b_j}v-\Psi(x_{b_j})v,$
that indeed is the kernel of $\delta_0$. As any other kind of de Rham complex, proving exactness is only a matter of symbol chasing; crucial point to prove it is that matrices $\Psi(x_i)$ commute each other.

With this free resolution, we can explicitly evaluate the $Ext^i$ functors, as cohomology of the following complex
$$0\to Hom_R(M,R)\xrightarrow{^t\delta_0} Hom_R(M_R,R)\xrightarrow{^t\delta_1} Hom_R(M_R^d,R) \xrightarrow{^t\delta_2}\ldots$$
$$\ldots\to Hom_R(M_R^{\binom{d}{d-1}},R)\xrightarrow{^t\delta_{d}}Hom_R(M_R^{\binom{d}{d}},R)\to 0,$$
where the maps are called $^t\delta_i$ because the matrices giving them are exactly the transposes of the one given in the previous complex. Now, as we already noticed, $Hom_R(M,R)$ is 0; moreover, we have that $$Hom_R(M_R^{\binom{d}{i}},R)\simeq M_R^{\binom{d}{i}}\simeq M_R^{\binom{d}{d-i}}$$ and up to sign changes the complex is precisely the same as before, except for using $^t\Psi(x_i)$ in the maps $^t\delta_i$. Therefore, in the end, we get that the sequence is exact, besides at the last spot, in which the cokernel (indeed, $Ext_R^d(M,R)$) is a module $^tM$ in which $R$ acts in a transpose way (by transpose matrices), that as we have seen before is exactly $Hom_{\mathbb{F}}(M,\mathbb{F})$.
\end{proof}

\begin{remark}
As a complex in the dual category of $S[\mathfrak{g}^*]$-modules (resp. $R(G)$-modules), the (derived) dual of any finite dimensional $\R$-vector space (resp. free $\Z$-module) has cohomology only supported in degree $d$; for the same reason, in some categories the definition of the dual includes also a shift in the derived category. This happens for instance for the category of coherent algebraic $D$-modules over an algebraic variety $Z$. It is interesting to remark that in this setting we have an analogue of Propositions \ref{ext-d} and \ref{ext-dm}: if $M$ is an integrable connections (that is, a coherent $D$-module that is also a coherent sheaf), than its dual as $D$-module (obtained as $Ext^{dim(Z)}(M,D_Z)$ because of the shift) is isomorphic to its dual as coherent sheaf, that is, $Hom(M,\mathcal{O}_Z)$; for all details, see \cite[Chapter 2.6]{Jap}.
\end{remark}

We can now focus on the modules $DM_{\mathcal{T}}(X)$ and $D_{\mathcal{T}}(X)$, which are free and (under the assumption that $\mathcal T$ contains all the cocircuits) have finite rank over $\Z$ and $\R$ respectively. We have:

\begin{theorem}\label{resume}
\begin{itemize}
.

\item[i)] $\x\begin{cases}Ext^d_{S[\mathfrak{g}^*]}(D_{\mathcal{T}}^*(X),S[\mathfrak{g}^*])\cong D_{\mathcal{T}}(X)\\Ext^i_{S[\mathfrak{g}^*]}(D_{\mathcal{T}}^*(X),S[\mathfrak{g}^*])=0 \text{ when }i\neq d\end{cases}$

\item[ii)] $\x\begin{cases}Ext^d_{S[\mathfrak{g}^*]}(D_{\mathcal{T}}(X),S[\mathfrak{g}^*])\cong D_{\mathcal{T}}^*(X)\\Ext^i_{S[\mathfrak{g}^*]}(D_{\mathcal{T}}(X),S[\mathfrak{g}^*])=0 \text{ when }i\neq d\end{cases}$

\item[iii)]$\x\begin{cases}Ext^d_{R(G)}(DM_{\mathcal{T}}^*(X),R(G))\cong DM_{\mathcal{T}}(X)\\Ext^i_{R(G)}(DM_{\mathcal{T}}^*(X),R(G))=0 \text{ when }i\neq d\end{cases}$

\item[iv)]$\x\begin{cases}Ext^d_{R(G)}(DM_{\mathcal{T}}(X),R(G))\cong DM_{\mathcal{T}}^*(X)\\Ext^i_{R(G)}(DM_{\mathcal{T}}(X),R(G))=0 \text{ when }i\neq d\end{cases}$
\end{itemize}
\end{theorem}

\begin{proof}

By Lemma \ref{fin-dim} $D_{\mathcal{T}}(X)$ has finite dimension over $\R$. Then by Lemma \ref{hom-d} also $D_{\mathcal{T}}^*(X)$ has finite dimension over $\R$. Therefore statements i) and ii) follow from Proposition \ref{ext-d}, since for every finite dimensional vector space $V$ we have that $Hom_\R((Hom_\R(V,\R), \R)\cong V$.

In the same way, by Lemmas \ref{fin-dim} and \ref{hom-dm}, $DM_{\mathcal{T}}(X)$ and $DM_{\mathcal{T}}^*(X)$ have finite rank over $\Z$. Then claims iii) and iv) follow from Proposition \ref{ext-dm}, since $Hom_\Z((Hom_\Z(M,\Z), \Z)\cong M$ for every free module of finite rank over $\Z$.
\end{proof}

\subsection{Duality for DPV modules}\label{duality-DPV}

Since DPV modules have not finite rank, we can not invoke Lemmas \ref{ext-d} and \ref{ext-dm}. In order to study duality of these objects, we need first to better examine their structure.

We will treat only the discrete case, since everything is the same in the differentiable setting.

From \cite{DPV1}, we get a canonical isomorphism
$$\mathcal{F}_{i}(X)/\mathcal{F}_{i+1}(X)\cong\bigoplus_{dim(\underline{s})=i}DM(X\cap\underline{s}).$$

Here $DM(X\cap\underline{s})$ is a submodule of $\mathcal{C}[\Gamma\cap\underline{s}]$, the module of $\Z$-valued functions on the lattice $\Gamma\cap\underline{s}$, which can be identified to functions in $\mathcal C[\Gamma]$ that are supported in $\Gamma\cap\underline{s}$.

By definition (see Section \ref{RT}) the corresponding group is $G/G_{\underline{s}}$, so that the dual space $DM^*(X\cap\underline{s})$ is a quotient of $R(G/G_{\underline{s}})$.

Furthermore, the quotient actually splits in a non-canonical way, meaning that we have isomorphisms (depending on a choice of some bases) $$\mathcal{F}_i(X)\cong\bigoplus_{dim(\underline{s})\geq i}DM(X\cap\underline{s}).$$

By taking on both sides the $R(G)$-modules generated, we get the isomorphisms
\begin{equation}\label{form-decomp}
\widetilde{\mathcal{F}}_i(X)\cong\bigoplus_{dim(\underline{s})\geq i}R(G)\otimes_{R(G/G_{\underline{s}})}DM(X\cap\underline{s}).
\end{equation}

Notice also that, still in a non canonical way, we have isomorphisms $G\cong G_{\underline{s}}\times G/G_{\underline{s}}$ and $R(G)\cong R(G_{\underline{s}})\otimes_{\Z} R(G/G_{\underline{s}})$, so that we have also
$$\widetilde{\mathcal{F}}_i(X)\cong\bigoplus_{dim(\underline{s})\geq i}R(G_{\underline{s}})\otimes_{\Z}DM(X\cap\underline{s}).
$$

Given this decomposition, we can prove the following Lemma.

\begin{lemma}\label{lem-DPV-dual} We have the following isomorphisms:
$$Ext^k_{R(G)}(\widetilde{\mathcal{F}}_i(X),R(G))\cong  \bigoplus_{dim(\underline{s})=k}R(G)\otimes_{R(G/G_{\underline{s}})}DM^*(X\cap\underline{s}) \text{ when } i\leq k\leq d$$
$$Ext^k_{R(G)}(\widetilde{\mathcal{F}}_i(X),R(G))=0 \text{ otherwise.}$$
\end{lemma}

\begin{proof}
Given the splitting in Formula (\ref{form-decomp}), we can work on each component separately; in this way, we just need to check that $$Ext^k_{R(G)}(R(G)\otimes_{R(G/G_{\underline{s}})}DM(X\cap\underline{s}),R(G))=R(G)\otimes_{R(G/G_{\underline{s}})}DM^*(X\cap\underline{s})$$ when $k=dim(\underline{s})$, and 0 otherwise. By applying Theorem \ref{resume}, everything follows from the sequence of isomorphisms
$$Ext^k_{R(G)}(R(G)\otimes_{R(G/G_{\underline{s}})}DM(X\cap\underline{s}),R(G))\cong Ext^k_{R(G/G_{\underline{s}})}(DM(X\cap\underline{s}),R(G))\cong$$
$$\cong R(G)\otimes_{R(G/G_{\underline{s}})}Ext^k_{R(G/G_{\underline{s}})}(DM(X\cap\underline{s}),R(G/G_{\underline{s}}))$$ 
where the first is just the change of rings theorem because $R(G)$ is flat over $R(G/G_{\underline{s}})$, and the second comes using again flatness of $R(G)$ and a free and finitely generated resolution of $DM(X\cap\underline{s})$ as $R(G/G_{\underline{s}})$-module (that exists because the ring is noetherian and the module finitely generated). Then, using Theorem \ref{resume} with the smaller group $G/G_{\underline{s}}$ concludes the proof.
\end {proof}

Thus for DPV modules several $Ext^i$s can be nonzero and must be taken into account. We give the following definition.

\begin{definition}\label{def-Fi*}
The \emph{(discrete) dual DPV modules} are the $R(G)$-modules
$$\widetilde{\mathcal{F}}_i^*(X)\doteq R(G)/\mathcal J^{\nabla}_i$$
where $\mathcal J^{\nabla}_i=(\nabla_{X\setminus\underline{s}})_{dim(\underline{s})<i}.$
\end{definition}

We believe that this is a good definition for many reasons.

First, the ideal $\mathcal J^{\nabla}_i$ is the annihilator of the $R(G)$ module $\widetilde{\mathcal{F}}_i(X)$; this can be shown using the decomposition (\ref{form-decomp}) above. Therefore $\widetilde{\mathcal{F}}_i^*(X)$ is defined is the same spirit as the dual DM module $DM^*(X)$.

Second, this algebraic notion of duality corresponds to a geometric duality: that is, the dual DPV modules appear as the equivariant K-theory of some spaces related with $\widetilde{\mathcal{F}}_i(X)$, as we will see in the following sections.

Third, these modules actually keep track of all information contained in the whole $Ext^*$. In fact, they admit filtration in which successive quotients are all the components of $Ext^*(\widetilde{\mathcal{F}}_i(X),R(G))$. In particular, we have the following proposition.

\begin{proposition}\label{DPV-filtr}
The module $\widetilde{\mathcal{F}}_i^*(X)$ admits a filtration $$\widetilde{\mathcal{F}}_{i,d}^*(X)\subset\widetilde{\mathcal{F}}_{i,d-1}^*(X)\subset\ldots\subset\widetilde{\mathcal{F}}_{i,i}^*(X)=\widetilde{\mathcal{F}}_i^*(X)$$ such that $$\widetilde{\mathcal{F}}_{i,j}^*(X)/\widetilde{\mathcal{F}}_{i,j+1}^*(X)=Ext^{d-j+i}_{R(G)}(\widetilde{\mathcal{F}}_{i}(X),R(G))$$
\end{proposition}

\begin{proof}
From the sequence of inclusions $\mathcal J_i^\nabla\subset \mathcal J_{i+1}^\nabla\subset\ldots\subset \mathcal J_{d}^\nabla$ we get the sequence of projections
$$\widetilde{\mathcal{F}}_i^*(X)\twoheadrightarrow\widetilde{\mathcal{F}}_{i+1}^*(X)\twoheadrightarrow\ldots\twoheadrightarrow\widetilde{\mathcal{F}}_d^*(X)\twoheadrightarrow0$$
and we take $\{\widetilde{\mathcal{F}}_{i,j}^*(X)\}$ as sequence of successive kernels. In this way we get isomorphisms $$\mathcal{F}_{i,j}^*(X)/\mathcal{F}_{i,j+1}^*(X)\cong\mathcal{F}_{d-j+i}^*(X)/\mathcal{F}_{d-j+i+1}^*(X)\cong \mathcal J_{d-j+i}^\nabla/\mathcal J_{d-j+i+1}^\nabla\cong$$
$$\cong(\nabla_{X\setminus\underline{s}})_{dim(\underline{s})<(d-j+i)}/(\nabla_{X\setminus\underline{s}})_{dim(\underline{s})<(d-j+i+1)}\cong$$
$$\cong\bigoplus_{dim(\underline{s})=(d-j+i)}\left(R(G)/(\nabla_{\underline{s}\backslash\underline{t}})_{_{\underline{t}\subsetneq\underline{s}}}\right)\nabla_{X\setminus\underline{s}}\cong$$
$$\cong\bigoplus_{dim(\underline{s})=(d-j+i)}R(G)\otimes_{R(G/G_{\underline{s}})}DM*(X\cap\underline{s})\cong Ext^{d-j+i}_{R(G)}(\widetilde{\mathcal{F}}_{i}(X),R(G))$$
where the last isomorphism comes from Lemma \ref{lem-DPV-dual}.
\end{proof}

In the same way, in the differentiable case we can give the following definition. The same considerations and statements apply.

\begin{definition}\label{def-Gi*}
The \emph{(differentiable) dual DPV modules} are the $S[\mathfrak{g}^*]$-modules
$$\widetilde{\mathcal{G}}_i^*(X)\doteq S[\mathfrak{g}^*]/\mathcal J^{\partial}_i$$
where $\mathcal J^{\partial}_i=(d_{X\setminus\underline{s}})_{dim(\underline{s})<i}.$
\end{definition}

\section{Recalls on equivariant K-theory}
We will briefly recall some notions about equivariant K-theory for a compact group $G$, and leave some more technical details to the appendix. The reader is suggested to refer to \cite{AtK} and \cite{Seg} for details and proofs.

\subsection{Definition}\label{defktheory}
Given a compact topological space $M$ with a continuous action of a compact Lie group $G$, one can consider \emph{equivariant complex vector bundles} on $M$, that is, complex vector bundles $E\to M$ with a $G$-action on the total space $E$, respecting the action on $M$ and acting linearly on fibers.
The \emph{equivariant $K$-theory} $K^0_G(M)$ of $M$ is the group of integer linear combination of isomorphism classes of such objects, with sum operation given by direct sum of vector bundles.

This group is naturally a ring endowed with tensor product, having zero element given by the 0-dimensional vector bundle and identity given by the bundle $\C\times M$ with the action of $G$ only on the second coordinate. Furthermore, we have a class of trivial bundles $V\times M$ where $V$ is a representation of $G$, so that we get an homomorphism $R(G)\to K_G(M)$ and hence giving $K_G^0(M)$ the structure of a $R(G)$-algebra with identity.

The definition for noncompact spaces is slightly more complicated, because K-theory is basically a theory with compact support; we leave it to the appendix. We then define $K_G^{-i}(M)=K_G^0(M\times \R^i)$. Because of Thom isomorphism (see \ref{thom} later) we will have $K_G^{-2}(M)=K_G^0(M)$. On one hand, this allows us to define $K_G^i(M)$ also for positive $i$, and on the other hand this let us focus just on $K_G^0$ (that will sometimes be denoted by simply $K_G$) and $K_G^1$, that contain all the information.

\subsection{Examples}\label{exmp}
\begin{enumerate}
\item If $M$ is a point, then $K_G(M)=R(G)$ , and $K^{1}_G(M)=0$.

\item if $M$ is given a trivial $G$-action, then $K_G(M)\cong R(G)\otimes K(M)$, where $K(M)$ is the usual $K-$theory of $M$;

\item if $M$ is given a free $G$-action, then we have an isomorphism of rings $K_G(M)\cong K(M/G)$.

\item if $G$ is abelian, and a subgroup $H$ stabilizes all points in $M$, then we have $K_G(M)\cong K_{G/H}(M)\otimes_{R(G/H)}R(G)$
\end{enumerate}

Proof of (4) is a slight modification of the classic proof of (2).

\section{Geometric realization of dual DM and DPV modules}\label{geom-dual}

Let $M_X^{fin}$, and more in general $M_X^{\geq i}$ $(0\leq i\leq d)$, be the spaces introduced in Section \ref{RT}.

In \cite{DPV2}, it is proved that $K_G^{2n-d}(M_X^{fin})\cong DM(X)$ and that $K_G^{2n-d+1}(M_X^{fin})\cong 0$; the proof of this fact relies on Atiyah's index of transversally elliptic operators, whose description goes way beyond the aim of this paper. Moreover, the actual setting of Atiyah's index is the space $T^*_GM_X^{fin}$ of covectors that are normal to $G$-orbits in $M_X^{fin}$; notice that on $M_X^{fin}$, because of the fact that orbits are al $d$-dimensional, this is a rank $2n-d$ real vector bundle; in general, this is not either necessarily a vector bundle.

In fact, the result in \cite{DPV2}, in its full generality, is the following:

\begin{theorem}[De Concini-Procesi-Vergne]\label{teor-F}
Atiyah's index gives an isomorphisms of $R(G)$-modules
$$K_G^0(T^*_GM_X^{\geq i})\cong\wt{\mathcal{F}}_i(X) \quad
\text{and} \quad K_G^1(T^*_GM_X^{\geq i})=0.$$
\end{theorem}

In \cite{DPV4} there is an analogue of these facts for equivariant cohomology with compact support:

\begin{theorem}[De Concini-Procesi-Vergne]\label{teor-G}
The infinitesimal index gives an isomorphism of graded $S[\mathfrak g^*]$-modules
$$H^*_{c,G}(T^*_GM_X^{\geq i})=\widetilde{\mathcal{G}}_{i}(X)$$
\end{theorem}

The theorem above is proved by introducing an analogue of the index of transversally elliptic operators (see \cite{DPV3}), the infinitesimal index; the correspondence is naturally with cohomology with compact support because of the ``compact support nature'' of equivariant K-theory.

In the same paper, there is also a calculation of the standard (meaning not with compact support) cohomology of  the spaces $M^{\geq i}_X$, which does not use any index theory and turns out to yield the dual modules described in Section \ref{duality}:

\begin{theorem}[De Concini-Procesi-Vergne]\label{teor-G^*}
There is an isomorphism of graded $S[\mathfrak g^*]$-algebras
$$H^*_{G}(T^*_GM_X^{\geq i})=\widetilde{\mathcal{G}}_{i}^*(X)$$
\end{theorem}
In particular for $i=d$ we have
$$H^*_{G}(M_X^{fin})=H^*_{G}(T^*_GM_X^{fin})=D^*(X)$$
since of course cohomology (with its natural grading) is preserved by the deformation retract.

Note that this theorem follows the Poincaré duality philosophy, for which (co)homology and cohomology with compact support give rise to modules that are dual each other.

We will now provide an analogue of Theorem \ref{teor-G^*} for equivariant K-theory. Instead of looking for a different definition of equivariant K-theory non involving compact support, we will compactify the spaces $M_X^{\geq i}$ (actually, deformation retracts of them) to find spaces to perform the same inductive process on.

The geometric idea behind this kind of compactification is the following: instead of removing a closed subspace from a compact space (then, loosing compactness), one removes a tubular neighborhood. In this way, the resulting space is still compact, but another property is lost: smoothness. More precisely, to perform differential geometry one has to enlarge the class of spaces to \emph{manifold with corners} (for instance, see \cite{Gai1} and \cite{Gai2}); anyway, this is not an issue, because we are not requiring any smoothness or boundarylessness condition (we are not either going to use the tangent space).

We now describe in more detail our construction. First, let us restrict to the unitary sphere $S_X\subset M_X$, to have a compact space to start with. By Formula (\ref{tolgochiusi}) $S_X\cap  M_X^{fin}$, and more generally $S_X\cap  M_X^{\geq i}$, are obtained by the compact space by removing some closed subsets $S_X\cap M_{\underline{s}}$. Now, for any such subset of $S_X$, we consider a small $G$-invariant tubular neighbourhood $$U_{\underline{s}}=\{(z_a)_{a\in X}\in S_X\text{ such that }\tq z_a|<\varepsilon\text{ if }a\notin\underline{s}\}$$ where $\varepsilon$ is a sufficiently small number (for the calculations we are going to do, it will be enough to assume $\varepsilon$ being any number smaller than $1/\sqrt{n}$, where $n$ is the cardinality of $X$) and finally we are ready to define
$$S_X^{\geq i}=S_X\backslash\bigcup_{\tiny \begin{array}{c} \underline{s}\in\rx\\ \text{dim}(\underline{s}) < i\end{array}}U_{\underline{s}}$$
(in fact, we can take the union only on the rational subspaces of dimension $i-1$) and to state the main result of this paper.

\begin{theorem}\label{teor-F^*}
We have \begin{center}$\x K^0_G(S_X^{\geq i})=\wt{\mathcal{F}}_i^*(X)$ and $K^1_G(S_X^{\geq i})=0$ for $1\leq i\leq d$ \end{center}
In particular,
\begin{center}$\x K^0_G(S_X^{fin})=DM^*(X)$ and $K^1_G(S_X^{fin})=0$\end{center}

\end{theorem}

Our proof will be based on a multiple induction. This requires to generalize it to a broader family of spaces; in order to do that, we will call \textit{simplicial} a set of nonzero rational subspaces $\mathcal{Q}$ that is closed under inclusions (if $\underline{s}\subset\underline{t}$ and $\underline{t}\in\mathcal{Q}$, then also $\underline{s}\in\mathcal{Q}$), and let $$S^{\mathcal{Q}}_X=S_X\backslash\bigcup_{\underline{s}\in\mathcal{Q}}U_{\underline{s}}.$$ Of course, the set of all proper rational subspaces, and more in general the sets $\{\text{dim}(\underline{s})<i\}$ are simplicial.

Then Theorem \ref{teor-F^*} is a corollary of the following:

\begin{theorem}\label{teor-simplicial}
If $\mathcal{Q}$ is an simplicial set of rational subspaces, then \begin{center}$\x K^0_G(S_X^{\mathcal{Q}})=
R(G)/(\nabla_{X\backslash\underline{s}})_{_{\underline{s}\in\mathcal{Q}}}$
and $K^1_G(S_X^{\mathcal{Q}})=0$.\end{center}
\end{theorem}

The beginning step of our induction is given by the following lemma.
\begin{lemma}\label{lemma11}
$K^0_G(S_X)\cong \widetilde{\mathcal{F}}_1^*(X)\cong R(G)/\nabla_X$, and $K^1_G(S_X)= 0$;
\end{lemma}
\begin{proof}
Let us consider the Mayer-Vietoris exact sequence
\begin{center}
\begin{tikzpicture}[description/.style={fill=white,inner sep=2pt}]
\matrix (m) [matrix of math nodes, row sep=3em,
column sep=2.5em, text height=1.5ex, text depth=0.25ex]
{ K_G^0(M_X\setminus \{\underline{0}\}) & K_G^0(M_X) & K_G^0(\{\underline{0}\}) \\
K_G^1(\{\underline{0}\}) & K_G^1(M_X) & K_G^1(M_X\setminus \{\underline{0}\}) \\ };
\path[->,font=\scriptsize]
(m-1-1) edge node[auto] {$  $} (m-1-2)
(m-1-2) edge node[auto] {$  $} (m-1-3)
(m-1-3) edge node[auto] {$  $} (m-2-3)
(m-2-3) edge node[auto] {$  $} (m-2-2)
(m-2-2) edge node[auto] {$  $} (m-2-1)
(m-2-1) edge node[auto] {$  $} (m-1-1);
\end{tikzpicture}
\end{center}

The bottom left and center elements of this sequence are 0 (because of Thom isomorphism), so we have to look at kernel and cokernel of $K_G^0(M_X)\to K_G^0(\{\underline{0}\})$; now, both this two modules are isomorphic to $R(G)$ (the latter being isomorphic to $R(G)$ as ring too), and the map is the multiplication by the Clifford class
$[\bigwedge^{odd}M_X]-\left[\bigwedge^{even}M_X\right]$
that by a straightforward calculation appears to be exactly $\nabla_X$; so, $K_G^0(M_X\setminus \{\underline{0}\})=0$ and $K_G^1(M_X\setminus \{\underline{0}\})=R(G)/\nabla_X$. But now, since $M_X\setminus \{\underline{0}\}=\R\times S_X$, we are done.
\end{proof}

The inductive step, on the other hand, will use the following geometric lemma.

\begin{lemma}\label{geomlemma}
Let $\mathcal{Q}$ be a simplicial set of rational subspaces, and $\underline{s}$ be a rational set that is minimal among those not in $\mathcal{Q}$ (that means, all proper rational subspaces of $\underline{s}$ belong to $\mathcal{Q}$). Let $S_{\underline{s}}$ be the unitary sphere in the vector subspace $M_{\underline{s}}$ of $M_X$ where the only nonzero coordinates are those of elements of $\underline{s}$. Then $S_X^{\mathcal{Q}}\cap U_{\underline{s}}$ is equivariantly homeomorphic to the normal bundle of $S_X^{\mathcal{Q}}\cap S_{\underline{s}}$ in $S_X^{\mathcal{Q}}$. More precisely, we have an equivariant homeomorphism
$$S_X^{\mathcal{Q}}\cap U_{\underline{s}}\cong (S_X^{\mathcal{Q}}\cap S_{\underline{s}}) \times M_{X\setminus\underline{s}}.$$ 
\end{lemma}

\begin{proof}
Let us describe the map explicitly; recall that we have
$$S_X^{\mathcal{Q}}\cap U_{\underline{s}}=\left\{(z_a)_{a\in X}\mid \begin{array}{c}
|z_a|<\varepsilon \ \text{if }a\notin\underline{s}\\
max_{a\notin \underline{t}}\{|z_a|\}\geq\varepsilon \ \text{for }\underline{t}\in\mathcal{Q}\end{array}\right\}$$
For a point in $z\in S_X^{\mathcal{Q}}$ let $$|z|^2_{\underline{s}}=\sum_{a\notin\underline{s}}|z_a|^2.$$
Then we define $\phi:S_X^{\mathcal{Q}}\cap U_{\underline{s}}\to (S_X^{\mathcal{Q}}\cap S_{\underline{s}}) \times M_{X\setminus\underline{s}}$ by

$$\phi(z)_a=\begin{cases}\frac{z_a}{\sqrt{1-|z|^2_{\underline{s}}}} \ \text{if }a\in\underline{s} \\
\frac{z_a}{\varepsilon^2-|z_a|^2} \ \text{if }a\notin\underline{s} \end{cases}.$$

It is very easy to check that the coordinates in $\underline{s}$ belong to a point of $S_X^{\mathcal{Q}}\cap S_{\underline{s}}$ (notice that because of $\varepsilon<1/\sqrt{n}$ the denominators are never zero), that the map is bijective, and $G$-equivariant, because $G$ does not change the absolute value. Hence, projecting onto the $S_X^{\mathcal{Q}}\cap S_{\underline{s}}$, we also get that $S_X^{\mathcal{Q}}\cap U_{\underline{s}}$ is a (equivariant) tubular neighborhood, and hence (equivariantly) homeomorphic to the normal bundle, of $S_X^{\mathcal{Q}}\cap S_{\underline{s}}$.
\end{proof}

We are now ready to prove Theorem \ref{teor-simplicial}. We will use some technical notions about equivariant $K$-theory that will be briefly explained in the appendix.

\subsection{Proof of Theorem \ref{teor-simplicial}}

Let us proceed by a double induction on the cardinality of $\mathcal{Q}$ and on the dimension of the group $G$; we start with $\mathcal{Q}=\emptyset$ and we will add one rational subspace at the time; this step will be based on the statement of the theorem for a lower dimensional group. When $\mathcal{Q}=\emptyset$, lemma \ref{lemma11} gives us the statement of the theorem.

Now, suppose we want to add a rational subspace $\underline{s}$ to $\mathcal{Q}$, such that all rational subspaces $\underline{t}$ of $\underline{s}$ already belong to $\mathcal{Q}$. On the geometric side, we have to remove from $S_X^{\mathcal{Q}}$ the open set $S_X^{\mathcal{Q}}\cap U_{\underline{s}}$, to get $S_X^{\mathcal{Q}\cup\{\underline{s}\}}$.
We will investigate the pushforward homomorphism
\begin{equation}\label{form-adm}
K_G^0(S_X^{\mathcal{Q}}\cap U_{\underline{s}})\to K_G^0(S_X^{\mathcal{Q}})
\end{equation}
that we wish to fill in the exact sequence of the inclusions $S_X^{\mathcal{Q}}\cap U_{\underline{s}}\hookrightarrow S_X^{\mathcal{Q}} \hookleftarrow S_X^{\mathcal{Q}\cup\{\underline{s}\}}$, to get the inductive step.

At first, as seen in Lemma \ref{geomlemma}, the set $S_X^{\mathcal{Q}}\cap U_{\underline{s}}$ is isomorphic to the normal bundle of the space $S_X^{\mathcal{Q}}\cap S_{\underline{s}}$ in $S_X^{\mathcal{Q}}$; hence, by Thom isomorphism, 
$$K^i_G(S_X^{\mathcal{Q}}\cap U_{\underline{s}})\cong K^i_G(S_X^{\mathcal{Q}}\cap S_{\underline{s}}).$$
Hence, let us focus on $S_X^{\mathcal{Q}}\cap S_{\underline{s}}$; by the fact that $\mathcal{Q}$ contains all rational subspaces of $\underline{s}$, we have
$$S_X^{\mathcal{Q}}\cap S_{\underline{s}}=S_{\underline{s}}\setminus\bigcup_{\underline{t}\subset\underline{s}}U_{\underline{t}}.$$
Considering now $H=ker(\underline{s})$ the kernel of all characters belonging
to $\underline{s}$ we have an action of $G/H$ on this space, because this space is contained in $M_{\underline{s}}$, so $H$ will stabilize every point of it; considering this action, this space is exactly $S_{\underline{s}}^{fin}$. So, by inductive hypothesis, we get that $K_{G/H}^1(S_X^{\mathcal{Q}}\cap S_{\underline{s}})=0$ and
$$K_{G/H}^0(S_X^{\mathcal{Q}}\cap S_{\underline{s}})=R(G/H)/(\nabla_{\underline{s}\backslash\underline{t}})_{_{_{\underline{t}\subset\underline{s}}}}.$$
By the fourth Example in Section \ref{exmp}, we get that $K_G^1(S_X^{\mathcal{Q}}\cap S_{\underline{s}})=0$ and $K_G^0(S_X^{\mathcal{Q}}\cap S_{\underline{s}})=R(G)/(\nabla_{\underline{s}\backslash\underline{t}})_{_{_{\underline{t}\subset\underline{s}}}}.$
We then get, going back through Thom isomorphism, $K_G^1(S_X^{\mathcal{Q}}\cap U_{\underline{s}})=0$ and $K_G^0(S_X^{\mathcal{Q}}\cap U_{\underline{s}})\cong R(G)/(\nabla_{\underline{s}\backslash\underline{t}})_{_{_{\underline{t}\subset\underline{s}}}}$, where the last is an isomorphism of $R(G)$-modules but \textit{not} of rings. Still by Thom isomorphism, and using that $S_X^{\mathcal{Q}}\cap U_{\underline{s}}\cong (S_X^{\mathcal{Q}}\cap S_{\underline{s}}) \times M_{X\setminus\underline{s}}$ from Lemma \ref{geomlemma}, the generator of $K_G^0(S_X^{\mathcal{Q}}\cap U_{\underline{s}})$ as $R(G)$-module is the Clifford element
$$\x\bigwedge^{odd}M_{X\setminus\underline{s}}\xrightarrow{c(E)}\bigwedge^{even}M_{X\setminus\underline{s}}.$$
Under the open pushforward into $K_G^0(S_X^{\mathcal{Q}})=R(G)/(\nabla_{X\backslash\underline{r}})_{_{_{\underline{r}\in\mathcal{Q}}}},$ this generator is sent into $\nabla_{X\setminus\underline{s}}$, obtained just writing the Clifford element in the difference of operators form.

Notice that the fact that $\nabla_{\underline{s}\setminus\underline{t}}=0$ in $K_G^0(S_X^{\mathcal{Q}}\cap S_{\underline{s}})$ is reflected into the fact that $$\nabla_{X\setminus\underline{s}}\cdot\nabla_{\underline{s}\setminus\underline{t}}=\nabla_{X\setminus\underline{t}}=0$$ in $K_G^0(S_X^{\mathcal{Q}})$, because $\underline{t}\in\mathcal{Q}$ for all proper subspaces of $\underline{s}$.

We can now investigate the Mayer-Vietoris sequence
\begin{center}
\begin{tikzpicture}[description/.style={fill=white,inner sep=2pt}]
\matrix (m) [matrix of math nodes, row sep=3em,
column sep=2.5em, text height=1.5ex, text depth=0.25ex]
{ K_G^0(S_X^{\mathcal{Q}}\cap U_{\underline{s}}) & K_G^0(S_X^{\mathcal{Q}}) & K_G^0(S_X^{\mathcal{Q}\cup\{\underline{s}\}}) \\
K_G^1(S_X^{\mathcal{Q}\cup\{\underline{s}\}}) & K_G^1(S_X^{\mathcal{Q}}) & K_G^1(S_X^{\mathcal{Q}}\cap U_{\underline{s}}) \\ };
\path[->,font=\scriptsize]
(m-1-1) edge node[auto] {$  $} (m-1-2)
(m-1-2) edge node[auto] {$  $} (m-1-3)
(m-1-3) edge node[auto] {$  $} (m-2-3)
(m-2-3) edge node[auto] {$  $} (m-2-2)
(m-2-2) edge node[auto] {$  $} (m-2-1)
(m-2-1) edge node[auto] {$  $} (m-1-1);
\end{tikzpicture}
\end{center}
where we know that the bottom right element is 0, and bottom center too by inductive hypothesis. The upper left homomorphism is, from what we have seen, just
$$R(G)/(\nabla_{\underline{s}\backslash\underline{t}})_{_{_{\underline{t}\subset\underline{s}}}}\xrightarrow{\cdot\nabla_{X\backslash\underline{s}}} R(G)/(\nabla_{X\backslash\underline{r}})_{_{_{\underline{r}\in\mathcal{Q}}}},$$
hence its kernel, that would be $K_G^1(S_X^{\mathcal{Q}\cup\{\underline{s}\}})$, is going to be zero, and its cokernel, that would be $K_G^0(S_X^{\mathcal{Q}\cup\{\underline{s}\}})$, is going to be $$R(G)/(\{\nabla_{X\backslash\underline{r}}\}_{_{_{\underline{r}\in\mathcal{Q}}}},\nabla_{X\setminus\underline{s}})=R(G)/(\{\nabla_{X\backslash\underline{r}}\}_{_{_{\underline{r}\in\mathcal{Q}\cup\{\underline{s}\}}}}),$$
that is exactly what we needed to prove.

\begin{remark}
In this way we found a geometric realization for $\widetilde{\mathcal{F}}_i^*(X)$ only for $i\geq 1$, because these are the only ones that can be got from simplicial sets; for $\widetilde{\mathcal{F}}^*(X)$, the extremal case, we haven't either defined a compact space to match with it; this is because we should consider also points that are fixed by the $G$ action, so the whole $M_X$, that doesn't retract onto the unit sphere $S_X$. But if we retract $M_X$ onto the origin, we immediately get
$$K_G^0(\{pt\})\cong\widetilde{\mathcal{F}}^*(X)\cong R(G)\ \text{and}\ K_G^1(\{pt\})=0$$
\end{remark}

How this theorem may lead to a more general statement about duality in equivariant K-theory is still a subject of investigation.

\section{The case of local DM modules}

As stated in Theorems \ref{teor-G} and \ref{teor-F}, the modules $DM(X)$ and $D(X)$ have a geometric realization in terms of the manifold $M_X^{fin}$, given by Atiyah's index and its cohomological analogue, the infinitesimal index.

Recall from sections \ref{loc-gen} the definitions of the local modules $DM_\Omega(X)$ and $D_\Omega(X)$, and from section \ref{RT} the submanifolds $M_X^{\Omega}\subset M_X^{fin}$ combinatorially defined in a similar way; 

We also know that these modules have interesting submodules, the local modules $DM_\Omega(X)$ and $D_\Omega(X)$, and that there are submanifolds $M_X^{\Omega}\subset M_X^{fin}$ combinatorially defined in a similar way; it is then natural to wonder if these submanifolds give geometric representations of the modules $DM_\Omega(X)$ and $D_\Omega(X)$.

\begin{conjecture}\label{cong-local}
There is an isomorphism of $R(G)$-modules
$$K_G^0(T^*_GM_X^\Omega)\cong DM_\Omega(X)$$
that is given by the Atiyah's index; furthermore, $K_G^1(T^*_GM_X^\Omega)=0$.
\end{conjecture}

A nice feature of this conjecture is that we already know from the theory of splines that the various spaces $DM_{\Omega}(X)$ are cyclic $R(G)$-modules generated by the quasi-polynomial $q_\Omega$ coming from the partition function; so, the inclusions $DM_{\Omega}(X)\hookrightarrow DM(X)$ give exactly a minimal set of generators of $DM(X)$ as $R(G)$-module. Then, this conjecture would provide a geometric analogue of this, giving the geometric support of these generators $q_\Omega$ considered inside $K_G^0(T^*_GM_X^{fin})$. 

Of course, the same can be stated for equivariant cohomology.

\begin{conjecture}\label{cong-local2}
There is an isomorphism of graded $S[\mathfrak{g}^*]$-modules
$$H_{c,G}^*(T^*_GM_X^\Omega)\cong D_\Omega(X)$$
that is given by the infinitesimal index.
\end{conjecture}

\section{Appendix: Further recalls on equivariant K-theory}

\subsection{Definition in the locally compact case}

To define equivariant K-theory in the locally compact case, we will need a different definition. Following \cite{Seg}, in a compactly supported fashion, equivariant K-theory is defined in the following way.

\begin{definition}
The equivariant K-theory of a locally compact $G$-space $M$ classifies objects of the kind $$\{E\xrightarrow{\phi}F\}$$
where $E$ and $F$ are equivariant vector bundles on $M$, and $\phi$ is an isomorphism outside a compact subspace of $M$.
\end{definition}

We are not going to explicitly describe the meaning of the word ``classifies''; the only things to keep in mind are that there are (quite obvious) notions of isomorphism and homotopy between such objects, and that objects like $\{E\xrightarrow{id}E\}$ are set to be equivalent to 0.

This definition agrees with the one given for compact spaces; in particular, on a compact space $M$, the correspondence is obtained sending the object $\{E\xrightarrow{\phi}F\}$ to the formal difference $[E]-[F]$; note that in the compact case the map $\phi$ may be \emph{nowhere} an isomorphism, and part of the proof of the equivalence is showing that we can move $\phi$ in an homotopic way to make it the zero function.

We have again an $R(G)$-algebra structure, but not necessarily with an identity element, and we do not have anymore a ring homomorphism $R(G)\to K_G(M)$ either.

Recall that we define $K_G^{-i}(M)=K_G(M\times\R^i)$, where $\R^i$ is given a trivial action of $G$, that again carries an $R(G)$-algebra structure (for $i=0$ we get the same object we just defined). As we have seen in \ref{defktheory}, all the information is actually contained in $K_G^{0}(M)$ and $K_G^{1}(M)$.

\subsection{Thom isomorphism}\label{thom}
If $E\xrightarrow{\pi}M$ is a $G$-vector bundle, then we have an isomorphism as $R(G)$-modules (but not as rings)
 $$K_G(M)\xrightarrow{Th}K_G(E).$$
If $M$ is compact, we can give an explicit description of the image of the identity element of $K_G(M)$, so that we can see it as a generator of $K_G(E)$ as $K_G(M)$-module.

We can consider $E$ as an equivariant vector bundle on $E$ itself (pulling back from the projection $E\to M$), and then consider all its exterior powers $\wedge^i E$ still as equivariant vector bundles on $E$. We will call $\wedge^{odd}E$ and $\wedge^{even}E$ the direct sum of all odd (resp. even) exterior powers of $E$; between this two vector bundles on $E$ we have a special map $c(E)$ (coming from the wedge product), called the \textit{Clifford map} (see \cite{DPV2}, pag. 795). So now, the data of
$$\x\bigwedge^{odd}E\xrightarrow{c(E)}\bigwedge^{even}E$$
is an element of $K_G(E)$, because $c(E)$ is an isomorphism everywhere except possibly on the zero section of $E$ (namely, $M$), which is indeed compact; this is the generator we were talking about.

\subsection{Functoriality and Mayer-Vietoris sequence}
Equivariant K-theory is a contravariant functor of $R(G)$-algebras for proper maps (by pullback), and a covariant functor of $R(G)$-modules for open embeddings (by extension).

Given $Z\xrightarrow{i}M$ closed equivariant embedding, calling $U=M\backslash Z\xrightarrow{j}M$ the open embedding, we have connecting homomorphisms $\delta$ giving the following exact sequence
\begin{center}
\begin{tikzpicture}[description/.style={fill=white,inner sep=2pt}]
\matrix (m) [matrix of math nodes, row sep=3em,
column sep=2.5em, text height=1.5ex, text depth=0.25ex]
{ K_G^0(U) & K_G^0(M) & K_G^0(Z) \\
K_G^1(Z) & K_G^1(M) & K_G^1(U). \\ };
\path[->,font=\scriptsize]
(m-1-1) edge node[auto] {$ j_* $} (m-1-2)
(m-1-2) edge node[auto] {$ i^* $} (m-1-3)
(m-1-3) edge node[auto] {$ \delta $} (m-2-3)
(m-2-3) edge node[auto] {$ j_* $} (m-2-2)
(m-2-2) edge node[auto] {$ i^* $} (m-2-1)
(m-2-1) edge node[auto] {$ \delta $} (m-1-1);
\end{tikzpicture}
\end{center}

\end{document}